\documentclass[a4,12pt]{amsart}
\usepackage{amssymb,amstext,amsmath,amscd,amsthm,amsfonts,enumerate,latexsym}
\usepackage{color}
\usepackage[dvipdfmx]{graphicx}
\usepackage[all]{xy}

\usepackage[dvipdfmx]{hyperref}

\theoremstyle{plain}
\newtheorem{thm}{Theorem}

\newtheorem*{thm*}{Theorem}
\newtheorem*{cor*}{Corollary}

\newtheorem{prop}[thm]{Proposition}

\newtheorem{lem}[thm]{Lemma}
\newtheorem{cor}[thm]{Corollary}

\newtheorem*{claim*}{Claim}

\theoremstyle{definition}

\newtheorem{ex}[thm]{Example}

\theoremstyle{remark}
\newtheorem{rem}[thm]{Remark}

\numberwithin{equation}{thm}

\def\Ext{\operatorname{Ext}}

\def\Ker{\operatorname{Ker}}

\def\Hom{\operatorname{Hom}}

\newcommand{\rmE}{\mathrm{E}}

\newcommand{\rmJ}{\mathrm{J}}
\newcommand{\rmK}{\mathrm{K}}

\newcommand{\fkp}{\mathfrak{p}}

\newcommand{\mapright}[1]{%
\smash{\mathop{%
\hbox to 1cm{\rightarrowfill}}\limits^{#1}}}

\newcommand{\mapleft}[1]{%
\smash{\mathop{%
\hbox to 1cm{\leftarrowfill}}\limits_{#1}}}

\def\Spec{\operatorname{Spec}}

\title[A criterion for reflexivity of modules]{A criterion for reflexivity of modules}

\author[Naoki Endo]{Naoki Endo}
\address{Department of Mathematics, Faculty of Science Division II, Tokyo University of Science, 1-3 Kagurazaka, Shinjuku, Tokyo 162-8601, Japan}
\email{nendo@rs.tus.ac.jp}
\urladdr{https://www.rs.tus.ac.jp/nendo/}

\author[Shiro Goto]{Shiro Goto}
\address{Department of Mathematics, School of Science and Technology, Meiji University, 1-1-1 Higashi-mita, Tama-ku, Kawasaki 214-8571, Japan}
\email{shirogoto@gmail.com}

\thanks{2020 {\em Mathematics Subject Classification.} 13C13, 13D30, 13C12.}
\thanks{{\em Key words and phrases.} reflexive module, $q$-torsionfree module, semi-local ring}

\thanks{The first author was partially supported by JSPS Grant-in-Aid for Young Scientists 20K14299. The second author was partially supported by JSPS Grant-in-Aid for Scientific Research (C) 21K03211.}

\begin{document}

\maketitle

\setlength{\baselineskip} {14.5pt}

\begin{abstract}
Let $M$ be a finitely generated module over a ring $\Lambda$. With certain mild assumptions on $\Lambda$, it is proven that $M$ is a reflexive $\Lambda$-module, once $M \cong M^{**}$ as a $\Lambda$-module.
\end{abstract}

\vspace{2em}

Let $\Lambda$ be a ring. For each left $\Lambda$-module $X$, let $X^* = \Hom_\Lambda(X,\Lambda)$ denote the $\Lambda$-dual of $X$. This note aims at reporting the following.

\begin{prop}\label{1}
Let $\Lambda$ be a ring and let $M$ be a finitely generated left $\Lambda$-module. Assume that one of the following conditions is satisfied.
\begin{enumerate}
\item[$(1)$] $\Lambda$ is a left Noetherian ring.
\item[$(2)$] $\Lambda$ is a semi-local ring, that is $\Lambda/\rmJ(\Lambda)$ is semi-simple, where $\rmJ(\Lambda)$ denotes the Jacobson radical of $\Lambda$.
\item[$(3)$] $\Lambda$ is a module-finite algebra over a commutative ring $R$.
\end{enumerate}
Then, $M$ is a reflexive $\Lambda$-module, that is the canonical map $M \overset{h_M}{\to} M^{**}$
is an isomorphism if and only if there is at least one isomorphism $M \cong M^{**}$ of $\Lambda$-modules.
\end{prop}

To show the above assertion, we need the following. This is well-known, and the proof is standard.

\begin{lem}\label{lemma}
Let $N$ be a right $\Lambda$-module and set $M = N^*$. Then the composite of the homomorphisms
$$
M \overset{h_M}{\to} M^{**} \overset{(h_N)^*}{\to} M
$$
equals the identity $1_{M}$.
\end{lem}

\begin{proof}[Proof of Proposition $\ref{1}$]
We have only to show the {\em if} part. Thanks to Lemma \ref{lemma}, we have a split exact sequence
$$
0 \to M \overset{h_M}{\to} M^{**} \to X \to 0
$$
of left $\Lambda$-modules, so that $M \cong M \oplus X$, since $M \cong M^{**}$. Therefore, we get a surjective homomorphism $\varepsilon : M \to M$ with $\Ker \varepsilon = X$. Hence, if Condition (1) is satisfied, then $X= (0)$, so that $M$ is a reflexive $\Lambda$-module. If Condition (2) is satisfied, then, setting $J = \rmJ(\Lambda)$, we get
$$
\Lambda/J \otimes_\Lambda M \cong (\Lambda/J \otimes_\Lambda M) \oplus (\Lambda/J \otimes_\Lambda X)
$$
whence $\Lambda/J \otimes_\Lambda X=(0)$ by Krull-Schmidt's theorem, so that $X = (0)$. Suppose that Condition (3) is satisfied. Then, $M \cong M \oplus X$ as an $R$-module, where $M$ is finitely generated also as an $R$-module. Consequently, for every $\fkp \in \Spec R$, we get $M_\fkp \cong M_\fkp \oplus X_\fkp$ as an $R_\fkp$-module, whence by the case where Condition (2) is satisfied, $X_\fkp=(0)$ for all $\fkp \in \Spec R$. Thus, $X = (0)$, as claimed.
\end{proof}

\begin{cor}\label{cor3}
Let $R$ be a commutative ring and $M$ a finitely generated $R$-module. If $M \cong M^{**}$ as an $R$-module, then $M$ is a reflexive $R$-module.
\end{cor}

\begin{rem}
Let $\Lambda$ be a ring and let $a,b \in \Lambda$ such that $ab=1$ but $ba \ne 1$. We then have the homomorphism 
$$\widehat{b} : {}_\Lambda \Lambda \to {}_\Lambda \Lambda, \ \ \ x \mapsto xb
$$ is surjective but not an isomorphism. Therefore, setting $X = \Ker \widehat{b}$, we get 
$$
{}_\Lambda \Lambda \cong {}_\Lambda \Lambda \oplus X.
$$ 
This example shows that $X$ does not necessarily vanish, even if $M \cong M \oplus X$ and $M$ is a finitely generated module. This example seems also to suggest that Proposition \ref{1} doesn't hold true without any specific conditions on $\Lambda$.
\end{rem}

Let us note one example in order to show how Corollary \ref{cor3} works at an actual spot. See \cite[p.137, the final step of the proof of (4.35) Proposition]{AB} also, where one can find a good opportunity of making use of it, from which the motivation for the present research has come.

\begin{ex}
Let $k[s,t]$ be the polynomial ring over a field $k$ and set $R=k[s^3, s^2t, st^2,t^3]$. Then $R$ is a normal ring and the graded canonical module $\rmK_R$ of $R$ is given by $\rmK_R=(s^2t,s^3)$. We set $I=(s^2t,s^3)$. Then, since $I$ is a reflexive $R$-module, but not $3$-torsionfree in the sense of Auslander-Bridger \cite[(2.15) Definition]{AB} (because $R$ is not a Gorenstein ring), we must have $\Ext_R^1(R:I,R) \ne (0)$  by \cite[(2.17) Theorem]{AB}. In what follows, let us check that $\Ext_R^1(R:I,R) \ne (0)$ directly.

First, consider the exact sequence
$$0 \to R \to R:I \to \Ext_R^1(R/I, R) \to 0
$$
induced from the sequence
$0 \to I \to R \to R/I \to 0$.
Taking the $R$-dual of it again, we get the exact sequence
$$ 0 \to R:(R:I) \to R \to \Ext_R^1(\Ext_R^1(R/I,R),R) \to \Ext_R^1(R:I,R) \to 0,$$
that is 
$$0 \to R/I \overset{\sigma}{\to} \Ext_R^1(\Ext_R^1(R/I,R),R) \to \Ext_R^1(R:I,R) \to 0.$$
Therefore, the homomorphism 
$$\sigma: R/I \to \Ext_R^1(\Ext_R^1(R/I,R),R) $$
should not be an isomorphism. Because 
$$
\Hom_{R/(f)}(\Hom_{R/(f)}(R/I,R/(f)),R/(f)) \cong  \Ext_R^1(\Ext_R^1(R/I,R),R)
$$
for every $0 \ne f \in I$,  thanks to Corollary \ref{cor3}, the assertion that $\sigma$ is not an isomorphism is equivalent to saying that $R/I$ is not a reflexive $R/(f)$-module for some $0 \ne f \in I$. In the following, we shall confirm  that $R/I$ is not a reflexive $R/(s^3)$-module. Before starting work, we would like to note here and emphasize that if we do not make use of Corollary \ref{cor3}, we must certify the above homomorphism $\sigma$ to be induced from the canonical map
$$
R/I \overset{h_{R/I}}{\to} \Hom_{R/(s^3)}(\Hom_{R/(s^3)}(R/I,R/(s^3)),R/(s^3)),
$$
which provably makes a tedious calculation necessary.

We set $T = R/(s^3)$ and $J = (\overline{s^2t},\overline{st^2})$ in $T$, where $\overline{*}$ denotes the image in $T$.  Notice that $\Hom_T(R/I,T) \cong (0):_TI = J$ and $\Hom_T(T/J, T) \cong (0):_TJ = (\overline{s^2t})$. Therefore, from the exact sequence
$$
0 \to J \to T \to T/J \to 0,
$$
we get the exact sequence
$$
0 \to (\overline{s^2t}) \to T \to \Hom_T(J,T) \to \Ext_T^1(T/J, T) \to 0,
$$
that is the exact sequence
$$
\ \ \ \ (\rmE) \ \ \ \ \ 0 \to R/I \to \Hom_T(J,T) \to \Ext_T^1(T/J, T) \to 0,
$$
which guarantees it suffices to show $\Ext_T^1(T/J, T)  \ne (0)$, since $\Hom_T(J,T) = \Hom_T(\Hom_T(R/I,T),T)$.
We now identify 
$$R = k[X,Y,Z,W]/\mathbf{I}_2\left(\begin{smallmatrix}
X&Y&Z\\
Y&Z&W\\
\end{smallmatrix}\right),
$$
where $k[X,Y,Z,W]$ denotes the polynomial ring over $k$, $\mathbf{I}_2(\Bbb M)$ stands for the ideal of $k[X,Y,Z,W]$ generated by the $2\times 2$ minors of a matrix $\Bbb M$,  and $X,Y,Z,W$ correspond to $s^3, s^2t,st^2,t^3$, respectively. We denote by $x,y,z,w$ the images of $X,Y,Z,W$ in $T$. Then, $T/J$ has a $T$-free resolution
$$
\ldots \to T^{\oplus 6} \overset{\left(\begin{smallmatrix}
y&z&0&0&0&0\\
-x&0&w&y&z&0\\
0&-x&-z&0&0&y\\
\end{smallmatrix}\right)}{\longrightarrow} T^3 \overset{\left(\begin{smallmatrix}
x&y&z\\
\end{smallmatrix}\right)}{\longrightarrow} T \to T/J \to 0.
$$
Taking the $T$-dual of the resolution, we have $\left(\begin{smallmatrix}
x\\
0\\
0
\end{smallmatrix}\right)
\in \Ker~ [T^{\oplus 3} \overset{\left(\begin{smallmatrix}
y&-x&0\\
z&0&-x\\
0&w&-z\\
0&y&0\\
0&z&0\\
0&0&y\\
\end{smallmatrix}\right)}{\longrightarrow} T^{\oplus 6}]$, but $\left(\begin{smallmatrix}
y\\
0\\
0
\end{smallmatrix}\right) \ne \alpha \left(\begin{smallmatrix}
x\\
y\\
z\\
\end{smallmatrix}\right)$
 for any $\alpha \in T$. Thus, $\Ext_T^1(T/J,T) \ne (0)$, so that  the exact sequence (E) shows $R/I$ is not a reflexive $T$-module. Hence, by Corollary \ref{cor3} the homomorphism
$$
\sigma: R/I \to \Ext_R^1(\Ext_R^1(R/I,R),R)
$$
is not an isomorphism. Thus, $\Ext_R^1(\rmK_R^*,R) \ne (0)$, and $\rmK_R$ is not $3$-torsionfree.
\end{ex}

\end{document}